\documentclass[11pt,reqno]{amsart}
\usepackage{amsthm}
\usepackage{amsmath,amssymb,txfonts,tipa,amsfonts,ascmac,mathrsfs,multicol,amscd,ascmac,fancybox}
\usepackage[dvipdfmx]{graphicx}
\usepackage[dvipdfmx]{color}
\usepackage[pagewise]{lineno}
\usepackage{amsaddr}
\usepackage[margin=1.2in]{geometry}
\usepackage{bm}

\usepackage{url}
\usepackage[initials]{amsrefs}
\usepackage{amsmath,amsthm,amscd}
\usepackage{enumerate}
\usepackage{amssymb}
\usepackage{mathrsfs}
\usepackage{comment}

\usepackage{mathtools}

\usepackage[latin1]{inputenc}
\usepackage{tikz}
\usetikzlibrary{shapes,arrows}
\usetikzlibrary{positioning}
\usetikzlibrary{intersections,calc}

\makeatletter
\@namedef{subjclassname@2010}{%
  \textup{2010} Mathematics Subject Classification}
\makeatother

\numberwithin{equation}{section}

\newtheorem{theorem}{Theorem}[section]

\newtheorem{proposition}[theorem]{Proposition}

\theoremstyle{definition}
\newtheorem{definition}[theorem]{Definition}
\newtheorem{remark}[theorem]{Remark}

\begin{document}

\baselineskip=17pt

\title[Lyapunov exponents for random maps]{Lyapunov exponents for random maps}

\author[F. Nakamura]{Fumihiko NAKAMURA}
\address[F. Nakamura]{Faculty of Engineering, Kitami Institute of Technology, Hokkaido, 090-8507, JAPAN}
\email[F. Nakamura]{nfumihiko@mail.kitami-it.ac.jp}

\author[Y. Nakano]{Yushi NAKANO}
\address[Y. Nakano]{Department of Mathematics, Tokai University, Kanagawa 259-1292, JAPAN}
\email[Y. Nakano]{yushi.nakano@tsc.u-tokai.ac.jp}

\author[H. Toyokawa]{Hisayoshi TOYOKAWA}
\address[H. Toyokawa]{Faculty of Engineering, Kitami Institute of Technology, Hokkaido, 090-8507, JAPAN}
\email[H. Toyokawa]{h\_toyokawa@mail.kitami-it.ac.jp}

\begin{abstract}
It has been recently realized  that for abundant dynamical systems  on a compact manifold, the set of points for which Lyapunov exponents fail to exist, called the Lyapunov irregular set, has positive Lebesgue measure.
In the present paper, we show that under any physical noise, the Lyapunov irregular set has zero Lebesgue measure and
 the number of such Lyapunov exponents is finite.
This result is a Lyapunov exponent version of Ara\'{u}jo's theorem on the existence and finitude of time   averages.
Furthermore, we numerically compute  the Lyapunov exponents for a surface flow with an attracting heteroclinic connection, which enjoys the Lyapunov irregular set of positive Lebesgue measure, under a physical noise.
This paper also contains the proof of the disappearance of Lyapunov irregular behavior on a positive Lebesgue measure set for a surface flow with an attracting homoclinic/heteroclinic connection under a non-physical noise.
\end{abstract}


\maketitle

\section{Introduction}
Let $f$ be a diffeomorphism on a closed  manifold $ M$ of dimension $d$. 
For a given continuous map $\psi : M\to \mathbb{R}$ the time average
\begin{equation}\label{eq:0317a}
\lim_{n\to\infty}\frac{1}{n}\sum_{j=0}^{n-1}\psi\circ f^j(x)
\end{equation}
sometimes fails to exist  for  every point $x$ in a set of  \emph{Lebesgue} positive measure (\cite{T,CV,KS,CYZ,HK,LR}), even though the Birkhoff ergodic theorem guarantees the existence of time averages for almost all points with respect to any \emph{invariant} measure.
The set of all points with no time average (for some $\psi$) is called the \emph{(Birkhoff) irregular set} in \cite{ABC} or the \emph{historic set} in \cite{R},
and 
 was also shown to be remarkably large in topological sense (e.g.~\cites{ABC,BLV1,BKNRS}) and thermodynamical sense (e.g.~\cites{BS,BLV2,EKL,PP}) for many dynamical systems.
However, Ara\'{u}jo showed in \cite{A} that surprisingly the Birkhoff irregular set disappears with respect to Lebesgue measure under physical noises (refer to Remark \ref{rmk:1} for detail; see  also \cite{A2001,A2003,AP2021,AT2005,N2017,KNS2019,C2019} for related works).

In the present paper, we focus on the \emph{Lyapunov irregular set} (\cite{ABC}), that is,  the set of all points $x\in M$ for which the Lyapunov exponent
\[
\lim_{n\to\infty}\frac{1}{n}\log\left\lVert Df^n(x)v\right\rVert
\]
does not exist for some vector $v\in\mathbb{R}^d\setminus\{0\}$ where the norm stands for the Euclidean norm.
We note that the Lyapunov irregular set is a null set with respect to any invariant probability measure 
 by the Oseledets multiplicative ergodic theorem. 
However, the Lyapunov irregular set, as well as the Birkhoff irregular set, for several maps was recently shown to be enormous in several senses \cites{BS,ABC,KLNS,OY,T2}.
In particular,  in 2008  Ott and Yorke \cite{OY} first gave an example of a dynamical system whose Lyapunov irregular set is of positive Lebesgue measure (as a classical surface flow with an \emph{attracting homoclinic  connection}, called a \emph{figure-8 attractor}), and it was 
  extended to abundant dynamical systems in 
    \cite{KLNS}. 
We emphasize that the Birkhoff irregular set and the Lyapunov irregular set do not have relations in general.
Indeed,  diffeomorphisms whose Birkhoff irregular set has   positive Lebesgue measure  but  Lyapunov irregular set has zero   Lebesgue measure were illustrated in \cites{CYZ}. 
Conversely, diffeomorphisms with the Lyapunov irregular set of positive measure and the Birkhoff irregular set of zero measure were exhibited in   \cite{OY,KLNS}  (see also \cite{F}).
Therefore, the following question naturally arises after the Ara\'{u}jo's theorem \cite{A}:
\begin{enumerate}
\item[]
Does the Lyapunov irregular set still have positive Lebesgue measure under physical noises?
\end{enumerate}
The purpose of this paper is to report that the answer is negative. 
Although the proof is elementary under the help of the paper \cite{A}, we believe that it is valuable to be  pointed out because of the importance of Lyapunov exponents in ergodic theory (\cite{BDV,KH,ASY,Viana}) and the previously-mentioned difference between Birkhoff irregular sets and Lyapunov irregular sets.

A novelty of the proof is
 to gather Ara\'ujo's obeservations in \cite{A} as a result for a \emph{deterministic} skew-product map (Theorem \ref{thm:Araujo}). 
Indeed   this skew-product formulation was already 
 realized by Ara\'ujo and played an indispensable role in his result (see e.g.~\cite[\S 2.4]{A}). 
However, it was not explicitly stated, and 
 a special emphasize in this paper would be helpful to prevent it from being
  ``buried'' in the long proof of \cite{A}.
 We also discuss possible applications of the skew-product formulation of Ara\'ujo's observations to zero Lebesgue measure of  irregular sets of local entropy and local dimension.

We will also see a numerical convergence of the Lyapunov exponent for a surface flow with an attracting heteroclinic connection under physical noises. This contrasts with the Ott-Yorke's numerical calculation in  \cite{OY}.
Furthermore, we prove the unique existence of the Lyapunov exponent for a surface flow with an attracting homoclinic/heteroclinic connection 
 under a \emph{non-physical} (impulsive) noise. 
 These descriptions may illustrate the meaning of our main theorem and imply a potential extension of Ara\'ujo's and our results.

\section{Main theorem}\label{s:main}

We first prepare notations 
from \cite{A}.
Let $M$ be a closed manifold of dimension $d$ equipped with the normalized Lebesgue measure $\mathrm{Leb}_M$. 
 A \emph{parametrized family} $f$ of $C^r$ diffeomorphisms on $ M$ with $r\geq 1$   is a differential map  from $B\times M$ to $M$ such that $f_t\equiv f(t, \cdot ) :M\to M$  is a $C^r$ diffeomorphism for all $t \in B$, where $B$ is the closed unit ball of a Euclidean space.
 Let $\mathcal B_B$ be the Borel $\sigma$-field of $B$ and $\mathrm{Leb}_B $  the normalized Lebesgue measure on $B$. 
 Let $(B^{\mathbb N}, \mathcal B_B^{\mathbb N}, \mathrm{Leb}_B  ^{\mathbb N})$ be the product space of the probability space $(B, \mathcal B_B, \mathrm{Leb}_B  )$.\footnote{
 In \cite{A}, instead of $(B^{\mathbb N}, \mathcal B_B^{\mathbb N}, \mathrm{Leb}_B  ^{\mathbb N})$, Ara\'ujo considered   the product space of a probability space $( B (a,\epsilon) , \mathcal B_{ B (a,\epsilon)}, \mathrm{Leb}_{B (a,\epsilon)}  )$ with some $a\in \mathrm{int} (B)$ and $\epsilon >0$ such that $B (a,\epsilon)\subset \mathrm{int} (B)$, where  $B(a,\epsilon) $ is the closed $\epsilon$-ball centered at $a$ and $\mathrm{int} (B)$ is the interior of $B$. 
 However, we can reduce it to our setting by considering $\tilde f( t, \cdot ) \coloneqq  f(  \epsilon (t+a),\cdot )$, $t\in B$.
}
For each $n\geq 1$, $\underline{t} =(t_1, t_2, \ldots ) \in B^{\mathbb N}$ and $x\in M$, we define  $f^n _{\underline t}(x) 
$ by 
\begin{equation}\label{eq:0318}
f^n_{\underline{t}}(x) =f_{t_n}\circ f_{t_{n-1}} \circ \cdots\circ f_{t_1}(x)
\end{equation}
and write $f^0_{\underline{t}}(x)=x$ for convenience. 
Notice that the $B$-valued random process $(\underline t =(t_1, t_2,\ldots ) \mapsto t_n)_{n\in \mathbb N}$ on $(B^{\mathbb N}, \mathcal B_B^{\mathbb N}, \mathrm{Leb}_B  ^{\mathbb N})$ is independently and identically distributed, so that we call $(n, \underline t, x) \mapsto f^n_{\underline t}(x)$ the \emph{i.i.d.~random dynamical system} induced by $f$ (refer to \cite{Arnold} 
 for 
the definition of
general random dynamical systems).
We also write $f_0^n$ for the usual $n$-th iteration of a single map $f_0: M\to M$ if it makes no confusion.
For each $x\in M$,  we let $\left(f^n_{(\cdot )}(x)\right) _* \mathrm{Leb}_B  ^{\mathbb N}$ denote the pushforward measure of $ \mathrm{Leb}_B  ^{\mathbb N}$ under  $f^n_{(\cdot )}(x) : B^{\mathbb N} \to M$, that is, $\left(\left(f^n_{(\cdot )}(x)\right) _* \mathrm{Leb}_B ^{\mathbb N}  \right)(A) =\mathrm{Leb}_B  ^{\mathbb N} \left(\left\{ \underline t\in B^{\mathbb N} \mid f^n_{\underline t }(x) \in A\right\} \right)$ for each Borel set $A\subset M$ (the measurability of $f^n_{(\cdot )}(x)  $  is ensured by \cite[Property 2.1]{A}). 
The following conditions are from \cite[Theorem 1]{A}.
\begin{definition}\label{d:phy}
The i.i.d.~random dynamical system induced by 
a parametrized family $f: B\times M \to M$ of $C^r$ diffeomorphisms is called \emph{physical} if 
there exists a $C^r$ diffeomorphism $f_0: M\to M$, an integer $n_0\geq 1$ and a real number $\xi_0>0$ such that for all $n\geq n_0$ and $x\in M$,
\begin{itemize}
\item[(A)] $\left\{ f^n_{\underline t} (x) \mid \underline t \in B^{\mathbb N} \right\}$ contains the $\xi_0$-ball centered at $f^n_0(x)$;
\item[(B)] $\left(f^n_{(\cdot )}(x)\right) _* \mathrm{Leb}_B  ^{\mathbb N}$ is absolutely continuous with respect to $\mathrm{Leb}_M$.
\end{itemize}
\end{definition}

See Theorem \ref{thm:Araujo} for the reason why we call such a random system ``physical'' from the viewpoint of physical measures.
Another justification for the name is given by the fact that the random map satisfying these conditions can be obtained from a Markov chain with smooth transition   probabilities (see e.g.~\cite{JKR2015}), which might appear often in physics.
Notice that in the deterministic case (i.e.~$f(t, \cdot )=f_0$ for each $t\in B$), the i.i.d.~random dynamical system induced by $f$ is not physical.
See \cite[Examples 1-4]{A} for examples of physical i.i.d.~random dynamical systems.
As for deterministic systems, we define the \emph{Lyapunov  irregular set} of the i.i.d.~random dynamical system induced by $f$ at $\underline t\in B^{\mathbb N}$ as the set of  points $x\in M$ for which 
\[
\lim_{n\to\infty}\frac{1}{n}\log\left\lVert Df^n_{\underline t}(x)v\right\rVert
\]
does not exist for some vector $v\in\mathbb{R}^d\setminus\{0\}$. 
Now we can state our main theorem:
\begin{theorem}\label{main}
Let $f: B\times M \to M$ be a parametrized family of  $C^r$ diffeomorphisms   with $r\geq 1$.
Suppose that the i.i.d.~random dynamical system induced by $f$ is physical.
Then, there exists  
 a    finitely many measurable sets   $\left( \mathcal V_i\right) _{1\leq i\leq \ell}$ of  $B^{\mathbb N}\times M$ satisfying   $(\mathrm{Leb}_B  ^{\mathbb N}\times \mathrm{Leb}_M)\left(\bigcup _{i=1}^\ell\mathcal V_i\right) =1$,   
real numbers  $\lambda _i^1 > \cdots > \lambda _i^{k(i)}$ and filtrations $\mathbb{R}^d=W_{(\underline{t},x)}^{i,1}\supset\cdots\supset W_{(\underline{t},x)}^{i,k(i)}\supset W_{(\underline{t},x)}^{i,k(i)+1}=\{0\}$ with $1\leq k(i)\leq d$  for  $1\leq i\leq \ell$, $(\underline t, x)\in \mathcal V_i$ such that 
\begin{equation}
\lim_{n\to\infty}\frac{1}{n}\log\left\lVert Df_{\underline{t}}^{n}(x)v\right\rVert=\lambda_i^j 
\end{equation}
for every $1\leq i\leq \ell$, $1\leq j\leq k(i)$, $(\underline t , x)\in \mathcal V_i$ and   $v\in W^{i,j}_{(\underline{t},x)}\backslash W^{i,j+1}_{(\underline{t},x)}$.

In particular, the Lyapunov irregular set at $\underline t$ has zero Lebesgue measure  for $\mathrm{Leb}_B  ^{\mathbb N}$-almost every $\underline t$.
\end{theorem}

\section{Proof}
Let $f: B\times M \to M$ be a parametrized family of $C^r$ diffeomorphisms.  
We define a skew-product map $S : B^{\mathbb N} \times M \to B^{\mathbb N} \times M$ by 
\[
S (\underline t , x) =\left((t_2, t_3,\ldots ), f_{t_1} (x)\right)\quad \text{for $\underline t=(t_1,t_2,\ldots )\in B^{\mathbb N}$ and $x\in M$},
\]
and call it the \emph{skew-product map induced by $f$.}
Note that
\begin{equation}\label{eq:0318b}
S ^n (\underline t, x) =\left((t_{n+1}, t_{n+2},\ldots ), f_{\underline t} ^n(x)\right).
\end{equation}
Recall that the \emph{generic set} $G(P)$ (or the \emph{ergodic basin})  of a probability measure $ P$ 
 by $S$ is the set of points $(\underline t, x)$ such that 
\[
\lim _{n\to \infty} \frac{1}{n} \sum _{j=0}^{n-1} \psi \circ S^j (\underline t, x) = \int _{B^{\mathbb N}\times M} \psi d P
\]
for every continuous function $\psi: B^{\mathbb N}\times M \to \mathbb R$.
Notice that $P(G(P)) =1$ when $P$ is an ergodic invariant probability measure by the Birkhoff ergodic theorem.  
Furthermore, we call a probability measure  $P$ \emph{physical} if $(\mathrm{Leb}_B  ^{\mathbb N}\times \mathrm{Leb}_M )(G(P)) >0$    (cf.~\cite{BDV}).
The following is a skew-product formulation of observations in Ara\'ujo \cite{A}.
We denote the support of a Borel probability measure $\nu$ by $\mathrm{supp} (\nu )$.
\begin{theorem}\label{thm:Araujo}
Let $S$ be the skew-product map induced by   a parametrized family $f: B\times M\to M$ of  $C^r$ diffeomorphisms with $r\geq 1$. 
Suppose that  the i.i.d.~random dynamical system induced by $f$ is physical. 
Then there are finitely many physical absolutely continuous ergodic invariant probability measures  $( P_i) _{1\leq i\leq \ell}$ of $S$ such that 
$(\mathrm{Leb}_B  ^{\mathbb N}\times \mathrm{Leb}_M)\left( \bigcup _{i=1}^\ell G(P_i)\right) =1$.
Furthermore, for every $1\leq i\leq \ell$ and $(\mathrm{Leb}_B  ^{\mathbb N}\times \mathrm{Leb}_M)$-almost every  $(\underline t, x) \in G(P_i)$, there exists an integer $k\geq 0$ such that $S^k (\underline t, x) \in \mathrm{supp}(P_i)$.
\end{theorem}

\begin{proof} 
According to  \cite[Section 3]{A}, we say that a finite pair of open sets $D=(U_0,\ldots , U_{k-1})$ is an  \emph{invariant domain} if $\overline{U_i} \cap \overline{U_j} =\emptyset$ for every $i\neq j$ and $f^n_{\underline t}(U_i) \subset U_{(n+i) \mod k}$ for every $n\in \mathbb N$, $\underline t\in B^{\mathbb N}$ and $0\leq i\leq k-1$. 
Let $\mathcal D$ be the set of all invariant domains and introduce a partial order $\preceq$ of $\mathcal D$ by saying $D\preceq D'$ for  $D=(U_0,\ldots , U_{k-1})$ and $D'=(U_0',\ldots , U_{k'-1}')$
if there are $i, i' \in \mathbb N$ such that  $D_{(i+j) \mod k} \subseteq D'_{(i'+j) \mod k'}$ for each $j$. 
A \emph{minimal} invariant domain  is an invariant domain  being minimal with respect to $\preceq$.
It was shown  in \cite[Section 6]{A} that the hypothesis (A) in Definition \ref{d:phy}  implies that there are at most finitely many minimal invariant domains $M_1, \ldots , M_\ell $. 
  
Moreover, in \cite[Section 7]{A} it was shown that the  hypothesis (B) in Definition \ref{d:phy} implies that there is a unique absolutely continuous ergodic stationary probability measure $\mu _i$ whose support coincides with  the closure of  $M_i$ for each $1\leq i\leq \ell$.
Here we mean by stationary that $\int _B \mu _i(f_t^{-1} A) \mathrm{Leb}_B  (dt ) = \mu _i(A)$ for each Borel set $A$,  and by  ergodic that if $\int _B \varphi \circ f_t(x) \mathrm{Leb}_B  (dt ) = \varphi (x)$ almost everywhere then   $\varphi $ is constant almost everywhere for each bounded measurable function $\varphi$.
Hence, it follows from \cite[Sections 2 and 7]{A} that $P_i\coloneqq  \mathrm{Leb}_B  ^{\mathbb N}\times \mu_i$ is a unique absolutely continuous ergodic invariant probability measure  whose support coincides with   the closure of  $B^{\mathbb N}\times M_i$.
Furthermore, it was proven that if one defines an open set $V_{i}(x)$ for $1\leq i\leq \ell$ and $x\in M$ by
\[
V_i(x)=\bigcup^\infty_{k=1} V_{i,k}(x), \quad V_{i,k}(x)=\left\{\underline{t}\in B^{\mathbb N} \mid  f^k_{\underline{t}}(x)\in M_i \right\},
\]
then $\left(V_i (x) \right)_{1\leq i\leq \ell}$ is a partition of $B^{\mathbb N}$ up to $\mathrm{Leb}_B  ^{\mathbb N}$-zero measure sets  for $\mathrm{Leb}_M$-almost every $x\in M$ (cf.~\cite[(18)]{A}), 
so that 
$(\mathrm{Leb}_B  ^{\mathbb N}\times \mathrm{Leb}_M)\left(\bigcup _{i=1}^\ell \mathcal V_i\right)=1 $ where $\mathcal V_i\coloneqq \bigcup _{x\in M} V_i(x) \times \{ x\}$. 

By the Birkhoff ergodic theorem, 
for each $1\leq i\leq N$,  there exists a $P_i$-full measure sets $\Gamma  \subset B^{\mathbb N}\times M_i$ such that for any $(\underline t, x) \in \Gamma$ and continuous function $\psi :B^{\mathbb N}\times M\to\mathbb R$, 
\[
\lim _{n\to\infty} \frac{1}{n} \sum _{j=0}^{n-1} \psi \circ S^j(\underline t,x) = \int \psi dP _i.
\]
Note that for each measurable set $A\subset \mathrm{supp} (P_i)$, if  $(\mathrm{Leb}_B  ^{\mathbb N}\times \mathrm{Leb}_M)(A) >0$, then one can find an integer $n$ such that $(\mathrm{Leb}_B  ^{\mathbb N}\times \mathrm{Leb}_M)(A_n) >0$, where $A_n\coloneqq\{  (\underline t,x) \in A\mid \frac{1}{n+1} \leq h_i(\underline t,x) < \frac{1}{n} \} $ ($n\geq 1$), $A_0\coloneqq\{  (\underline t,x) \in A\mid h_i(\underline t,x) \geq 1 \} $ and $h_i$ is the Radon-Nikod\'{y}m derivative of $P_i$. 
Thus we get
$
P_i(A) \geq (\mathrm{Leb}_B  ^{\mathbb N}\times \mathrm{Leb}_M)(A_n)/(n+1) >0.
$ 
By applying the contraposition of this implication to the $P_i$-zero measure set $( B^{\mathbb N}\times M_i) \setminus \Gamma$, we get $(\mathrm{Leb}_B  ^{\mathbb N}\times \mathrm{Leb}_M)(( B^{\mathbb N}\times M_i) \setminus \Gamma)=0$.
Hence, observing $\mathcal V_i = \bigcup _{k=0}^\infty S^{-k}(B^{\mathbb N} \times M_i) $, we have
\begin{equation}\label{eq:0205}
\left(\mathrm{Leb}_B  ^{\mathbb N}\times \mathrm{Leb}_M\right) \left(\mathcal V_i \setminus \bigcup _{k=0}^\infty S^{-k}\Gamma \right)
=0.
\end{equation}
Note also that if $(\underline s, y) = S^k(\underline t, x)$ with some $k\geq 0$, then for any continuous function $\psi :B^{\mathbb N}\times M\to \mathbb R$ we have 
\begin{equation}\label{eq:0203b}
\left\vert \frac{1}{n} \sum _{j=0}^{n-1} \psi \circ S^j(\underline t, x) - \frac{1}{n} \sum _{j=0}^{n-1} \psi \circ S^j(\underline s, y) \right\vert \leq   \frac{2k  }{n}\sup _{(\underline t, x)\in B^{\mathbb N}\times M}\vert \psi (\underline t, x) \vert  \to 0
\end{equation}
as $n\to\infty$
(recall that both $B$ and  $M$ are   compact).
Hence, for $(\mathrm{Leb}_B  ^{\mathbb N}\times \mathrm{Leb}_M)$-almost every $(\bar t, x)\in \mathcal V_i$, we get $\lim _{n\to\infty} \frac{1}{n} \sum _{j=0}^{n-1} \psi \circ S^j(\underline t,x) = \int \psi dP _i$.
That is, $\mathcal V_i \subset G(P_i)$ up to $(\mathrm{Leb}_B  ^{\mathbb N}\times \mathrm{Leb}_M)$-zero measure sets, and thus, we conclude $(\mathrm{Leb}_B  ^{\mathbb N}\times \mathrm{Leb}_M)\left(\bigcup _{i=1}^\ell G(P_i)\right)=1$.
This completes the proof.
\end{proof}

We now prove  Theorem \ref{main}. 
Due to Theorem \ref{thm:Araujo}, there are finitely many  physical absolutely continuous ergodic invariant probability measures $P_1, \ldots , P_\ell$ of $S$ such that the union of the generic set of $P_i$'s covers 
 $B^{\mathbb N}\times M$ up to $\mathrm{Leb}_B  ^{\mathbb N}\times \mathrm{Leb}_M$ zero measure sets. 
Define a measurable map $A^{(n)}: B^{\mathbb N}\times M\to GL(d,\mathbb R)$ for $n\geq 0$ by
\[
A^{(n)}(\underline t, x)= Df_{\underline t}^n(x).
\]
Then it follows from \eqref{eq:0318} and \eqref{eq:0318b}  that
\begin{eqnarray*}
A^{(n)}(\underline t, x) &=& Df_{t_n} \left(f_{\underline t}^{n-1}(x)\right) \cdot Df_{t_{n-1}} \left(f_{\underline t}^{n-2}(x)\right)\cdots Df_{t_1} (x) \\
&=& A^{(1)} \circ S^{n-1}(\underline t, x)\cdot A^{(1)} \circ S^{n-2}(\underline t, x) \cdots A^{(1)} (\underline t, x),
\end{eqnarray*}
and so, $(A^{(n)})_{n\geq 0}$ satisfies the cocycle property for $S$, that is,
\[
A^{(n+m)}(\underline t, x) = A^{(n)} \circ S^m(\underline t, x) \cdot A^{(m)}(\underline t ,x)
\]
for every $n, m\geq 0$ and $(\underline t, x) \in B^{\mathbb N}\times M$. 
Furthermore, since $f: B\times M\to M$ is a differential map and $B$ is compact, the function 
\begin{equation}\label{eq:0203a}
(\underline t, x) \mapsto \max\left\{ \log \left\lVert A^{(1)}(\underline t, x)\right\rVert , 0\right\} :B^{\mathbb N}\times M\to \mathbb R
\end{equation}
is bounded, in particular, $P_i$-integrable for each $1\leq i\leq \ell$.

Hence, we can apply  the Oseledets multiplicative ergodic theorem to the cocycle $(A^{(n)})_{n\geq 0}$ over the ergodic measure-preserving system $S$ on $(B^{\mathbb N}\times M, P_i)$ for each $1\leq i\leq \ell$: 
There are  a $P_i$-full measure set $\Gamma$,   real numbers $\lambda_i^1>\cdots>\lambda_i^{k(i)}$ 
and a filtration $\mathbb{R}^d=W_{(\underline{t},x)}^{i,1}\supset\cdots \supset W_{(\underline{t},x)}^{i,k(i)}\supset W_{(\underline{t},x)}^{i,k(i)+1}=\{0\}$ with some integer $1\leq k(i) \leq d$ such that 
 $A^{(1)}(\underline{t},x)W_{(\underline{t},x)}^{i,j}=W_{S(\underline{t},x)}^{i,j}$ and
\[
\lim_{n\to\infty}\frac{1}{n}\log \left\lVert Df_{\underline{t}}^n(x)v\right\rVert = \lim_{n\to\infty}\frac{1}{n}\log \left\lVert A^{(n)}(\underline{t},x)v\right\rVert=\lambda_i^j\quad\text{for all $v\in W^{i,j}_{(\underline{t},x)}\backslash W^{i,j+1}_{(\underline{t},x)}$}
\]
for all $(\underline{t},x) \in \Gamma$.
On the other hand, $G(P_i)$ coincides with $ \bigcup _{k=0}^\infty S^{-k}(B^{\mathbb N} \times M_i)$ up to $(\mathrm{Leb}_B  ^{\mathbb N}\times \mathrm{Leb}_M)$-zero measure sets  by virtue of Theorem \ref{thm:Araujo},  so by repeating  the argument for \eqref{eq:0205} we get
\[
\left(\mathrm{Leb}_B  ^{\mathbb N}\times \mathrm{Leb}_M\right) \left(G(P_i) \setminus \bigcup _{k=0}^\infty S^{-k}\Gamma \right)
=0.
\]
Observe also that if $(\underline s, y) =S^k(\underline t, x)$ with some $k\geq 0$, then for each nonzero vector $v$, by letting $w\coloneqq Df^k_{\underline t}(x) v$, we have
\begin{equation}\label{eq:0203c}
\left\vert \frac{1}{n} \log \left\Vert Df^n_{\underline t}(x)v\right\Vert  -  \frac{1}{n} \log \left\Vert Df^n_{\underline s}(y)w\right\Vert \right\vert \leq   \frac{2k}{n} \sup _{(\underline t, x)\in B^{\mathbb N}\times M}\log \left\Vert A^{(1)} (\underline t, x)\right\Vert \to 0
\end{equation}
as $n\to\infty$ (recall \eqref{eq:0203a}). Therefore,  
since $(\mathrm{Leb}_B  ^{\mathbb N}\times \mathrm{Leb}_M)\left(\bigcup _{i=1}^\ell G(P_i)\right) =1$ due to Theorem \ref{thm:Araujo}, 
by letting $\mathcal V_i \coloneqq G(P_i)$, this completes the proof of Theorem \ref{main}.

\begin{remark}\label{rmk:1}
We emphasize that it is essential in the proof of Ara\'ujo's and our theorems that, as Theorem \ref{thm:Araujo} stated,  the orbit of a point in the generic set of $P_i$ hits the support of $P_i$ in a \emph{finite time}. 
This property (together with \eqref{eq:0203b} and \eqref{eq:0203c}) allows us to apply classical ergodic theorems for $(\mathrm{Leb}_B  ^{\mathbb N}\times \mathrm{Leb}_M)$-almost every $(\underline t, x)$.
This is contrastive to the deterministic case in which the orbit of a point in a positive Lebesgue measure set may not hit  the support of any ergodic invariant measure in a finite time (as examples in Sections \ref{s:ne} and \ref{s:npn}), and thus the Birkhoff/Lyapunov irregular sets may have positive Lebesgue measure. 

Furthermore, 
we remark that the (thermodynamical) largeness of the irregular set of local entropy
 and the irregular set  of local dimension were discussed in \cite{BS}.
As Birkhoff and Lyapunov irregular sets, these irregular sets also have zero measure   with respect to    invariant measures in an appropriate setting  due to, e.g., Shannon-McMillan-Breiman theorem for local entropy and  Ledrappier-Young or Barreira-Pesin-Schmeling theorem for local dimension (refer to \cite{BW2006} and references therein).
Therefore, it seems natural to hope that Theorem \ref{thm:Araujo} is   a key tool to establish zero Lebesgue measure of these irregular sets under physical noise. 
\end{remark}

\section{Numerical example}\label{s:ne}

In this section, we experiment on the computation of a finite time flow Lyapunov exponent for a system with an attracting heteroclinic connection under additive noises which is physical perturbation. 
The model (without any perturbation) was originally raised by Ott and Yorke in \cite{OY} and was illustrated via numerical simulation to exhibit absence of the Lyapunov exponent.
They considered the following differential equation over $[0,\pi]^2$,
\begin{eqnarray}\label{eqot}
\begin{cases}
\dot{x}= \cos y \sin x - a \cos x \sin x\quad(\eqqcolon f_1(x,y))\\
\dot{y}=-\cos x \sin y - a \cos y \sin y\quad(\eqqcolon f_2(x,y))
\end{cases}\quad\quad(a\in(0,1)),
\end{eqnarray}
and calculated numerically the finite time flow Lyapunov exponent $\lambda(t,{\bm x},{\bm f}({\bm x}))$ where ${\bm f}=(f_1,f_2)$ and ${\bm x}=(x,y)$. 
Here the finite time flow Lyapunov exponent is defined by
\[
\lambda(t,{\bm x},v)=\frac{1}{t}\log\left\Vert D\varphi({\bm x},t)v \right\Vert,
\]
where $\varphi({\bm x},t)$ is the continuous flow induced by \eqref{eqot}.  
Their numerical result showed an oscillating graph of $\lambda(t,{\bm x},{\bm f}({\bm x}))$ that means  non-existence of the Lyapunov exponent (see also Figure \ref{fig1}).
In what follows, we perturb this system with additive noises from the interval $[-\epsilon,\epsilon]$, for some $\epsilon>0$, and show the numerical evidence of Theorem \ref{main} that the amplitude of the oscillation of the finite time flow Lyapunov exponent decreases and vanishes.

\begin{remark}
For time averages, it was proven by Takens \cite{T} that any surface flow with dissipative heteroclinically connected two saddle fixed points, which is called the Bowen flow and whose dynamics is quite similar to that of \eqref{eqot}, has the Birkhoff irregular set of positive Lebesgue measure (while the mechanism for the absence of time averages of the Bowen flow is different from the mechanism for the absence of Lyapunov exponents; see, e.g., the remark after Proposition 1.1 in \cite{KLNS}).
On the other hand,  it was shown by Ara\'ujo \cite[Section 10]{A} (see also \cite[Example 2]{A2001}) that the Birkhoff irregular set of the Bowen flow turns out to be of zero Lebesgue measure under physical noise.
\end{remark}

From the argument in \cite{OY}, we can compute the finite time flow Lyapunov exponent as
\begin{eqnarray}\label{flowlyap}
\lambda\left(t,{\bm x}_0,{\bm f}({\bm x}_0)\right)=\frac{1}{t}\log\sqrt{f_1\left((x(t),y(t)\right)^2+f_2\left(x(t),y(t)\right)^2}
\end{eqnarray}
where $(x(t),y(t))$ is the solution of \eqref{eqot} with an initial condition $(x(0),y(0))={\bm x_0}$.
We immediately find that $f_1((x(t),y(t))^2+f_2(x(t),y(t))^2$ is $0$ at $(x,y)=(0,0),(0,\pi),(\pi,0),(\pi,\pi)$ and is $1$ at $(x,y)=(0,\pi/2),(\pi/2,0),(\pi,\pi/2),(\pi/2,\pi)$.
Note that the four corners are dissipative saddles and the center $(\pi/2,\pi/2)$ is a source, and the four sides of the square form the heteroclinic connection.
In order to perform the numerical integration, the coordinate transformation
$z=\tan(x-\pi/2)$ and $w=\tan(y-\pi/2)$ is used, which enables one to avoid the error problem associated with computing vector fields near the equilibria. We then obtain the differential equation on $\mathbb{R}^2$
\begin{eqnarray}
\begin{cases}
\dot{z}=  w\sqrt{\frac{1+z^2}{1+w^2}} +az\\
\dot{w}= -z\sqrt{\frac{1+w^2}{1+z^2}} +aw.
\end{cases}
\end{eqnarray}

Set an initial point by $(x_0,y_0)=((3+\pi)/2,(3+\pi)/2)$ which is, according to our coordinate transformation, equivalent to $(z_0,w_0)=(\tan (3/2),\tan (3/2))$.
In Figure \ref{fig1}, we illustrate the numerical result of the finite time flow Lyapunov exponent with respect to this initial point for the system \eqref{eqot} with additive noises by using the following algorithm.

(STEP 1) Set an initial point $(z[0],w[0])=(\tan(3/2),\tan(3/2))$.

(STEP 2) Compute $(z[t],w[t])$ from $(z[t-1],w[t-1])$ by Runge-Kutta-Fehlberg method for $t=1,\cdots,T$. 

(STEP 3) Calculate $x[t]=\arctan(z[t])+\pi/2+\epsilon \omega_t$ and $y[t]=\arctan(w[t])+\pi/2+\epsilon \xi_t$, where $\omega_t$ and $\xi_t$ are independent random variables each of which has the uniform distribution $1_{[-1,1]}$, that is, $\omega_t$ and $\xi_t$ are randomly chosen uniformly from $[-1,1]$.

(STEP 4) Replace $(z[t],w[t])$ by $(\tan(x[t]-\pi/2),\tan(y[t]-\pi/2))$.

Repeating (STEP 2)--(STEP 4) for $t=1,\cdots,T$, we can calculate $\lambda\left(t,{\bm x}_0,{\bm f}({\bm x}_0)\right)$ by Eq.~\eqref{flowlyap} with $x[t]$ and $y[t]$.
In our simulation, we set parameters of step-size $h=0.001$, $T=10^7$, $a=0.03$ and the noise level $\epsilon=0.0001$. 

\begin{figure}[htbp]
\begin{center}
 \includegraphics[width=14cm,bb=0mm 0mm 200mm 105mm,clip]{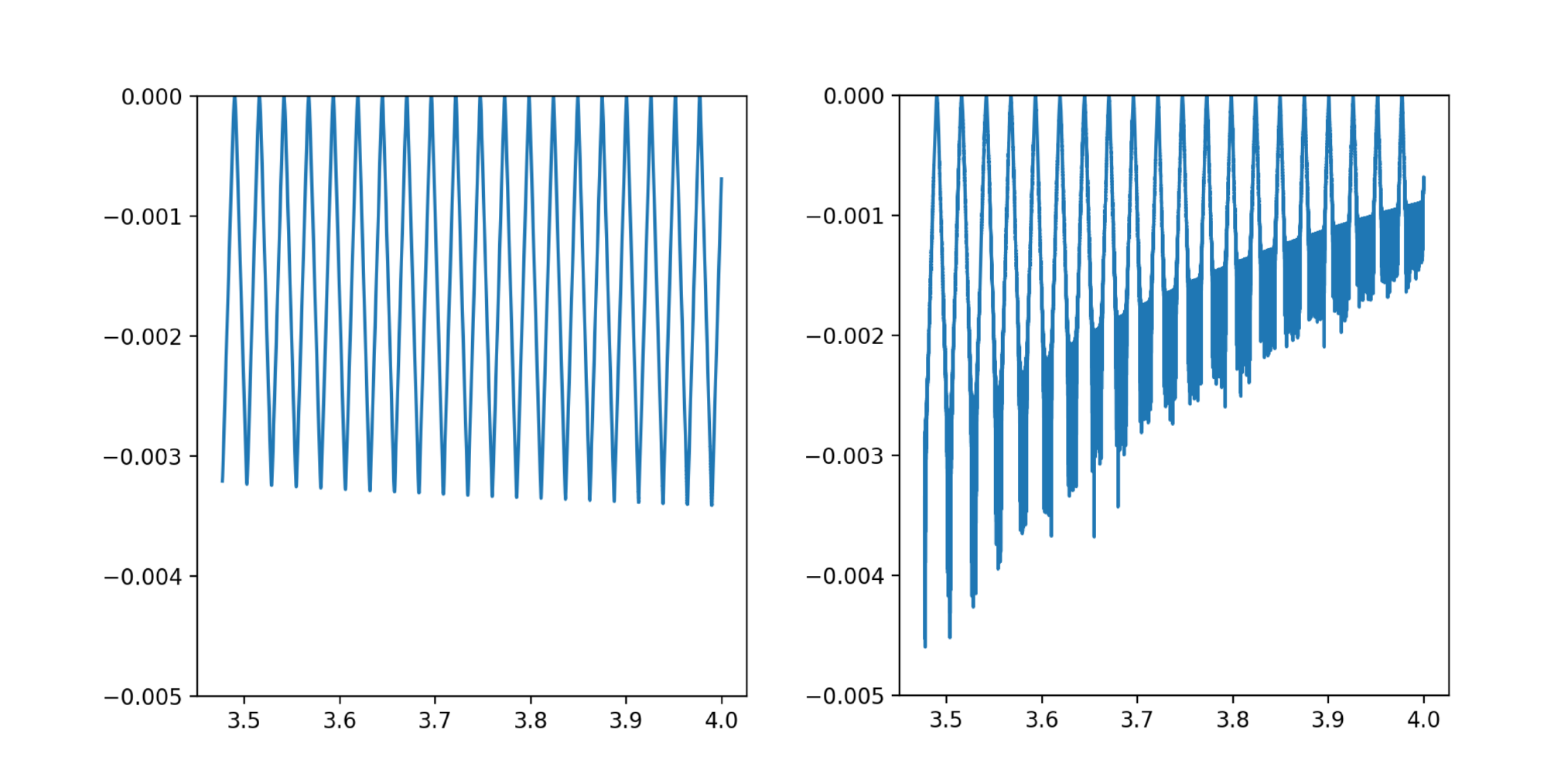}
 \end{center}
 \caption{The finite time flow Lyapunov exponent for the system \eqref{eqot} are illustrated without noise (left), with additive noise (right). They are plotted by $\log$ time scale.}
 \label{fig1}
\end{figure}

\begin{remark}
When we consider the perturbed system, the orbit may go out from $[0,\pi]^2$. However, a similar behavior arises in each region 
$\left[k\pi,(k+1)\pi\right]\times\left[\ell\pi,(\ell+1)\pi\right]$ with $k,\ell\in\mathbb{Z}$ because ${\bm f}$ is a periodic function. Just the directions of the rotation are different depending on whether $k+\ell$ is odd or even. Thus, more precisely, we use the following replacement of (STEP 3) in our experiment:
if $(x[t],y[t])\in \left[k\pi,(k+1)\pi\right]\times\left[\ell\pi,(\ell+1)\pi\right]$,
 \begin{eqnarray*}
 \left(-\tan(x[t]-\pi/2-k\pi),-\tan(y[t]-\pi/2-\ell\pi)\right)&\quad\text{$k+\ell$ is even,}\\
 \left(\tan(x[t]-\pi/2-k\pi),\tan(y[t]-\pi/2-\ell\pi)\right)&\quad\text{$k+\ell$ is odd.}
 \end{eqnarray*}
\end{remark}

\section{Non-physical noises}\label{s:npn}

In this section, we will prove that the Lyapunov irregular set of positive Lebesgue measure for a simple figure-8 type diffeomorphism in \cite{KLNS} vanishes when we add certain non-physical perturbation of impulsive type to the system.
Our random perturbation of impulsive type is not physical so that the result in this section is not consequence of Theorem \ref{main}, but Theorem \ref{propA} suggests possibility of generalization of Ara\'{u}jo's and our results.
One may find examples of random perturbations of impulsive type and its physical motivation in \cites{DPZ,GV}.

The model dealt with in this section was introduced in \cite{GGS} as follows.
Fix a constant $\kappa >1$ and numbers $a$, $b$ with $1<a<b<\kappa$.
Let $I=[a,b]$ and ${\bm H}$ denote the map $\mathbb{R}^2\ni(x,y)\mapsto \left(\kappa ^{-2}x,\kappa y\right)$.
We let $S_n=I\times\kappa ^{-n}I$ and $U_n=\kappa ^{-n}I\times I$, so that
\[
{\bm H}^n(S_n)=U_{2n}\quad\text{and}\quad {\bm H}^n:S_n\to U_{2n} \text{ is a diffeomorphism.}
\]
Furthermore, let ${\bm R}$ be the rotation of angle $\pi/2$ with center $\left(\frac{a+b}{2},\frac{a+b}{2}\right)$, i.e.,
\[
{\bm R}(x,y)=(a+b-y,x)\quad\text{for }(x,y)\in\mathbb{R}^2.
\]
Recall that a diffeomorphism on the plane is said to be \emph{compactly supported} if it equals to the identity outside of some ball centered at the origin.
The existence of such a simple model of a compactly supported diffeomorphism with an attracting homoclinic connection was proven in \cite{GGS}.

\begin{proposition}[Proposition 3.4 of \cite{GGS}]\label{propGGS}
There exists a compactly supported $C^{\infty}$ diffeomorphism ${\bm f}:\mathbb{R}^2\to\mathbb{R}^2$ together with positive integers $n_0$ and $k_0$ such that the following hold:
\begin{enumerate}
\item[(a)]
There is a neighborhood $V$ of the origin ${\bm o}=(0,0)\in\mathbb{R}^2$ such that
\[
\bigcup_{n\ge n_0}\bigcup_{0\le\ell\le n} {\bm f}^{\ell}(S_n)\subset V\quad\text{and}\quad {\bm f}\mid_{V}={\bm H}.
\]
In particular, ${\bm o}$ is a hyperbolic fixed point of saddle type for ${\bm f}$.
\item[(b)]
The stable manifold of ${\bm o}$ coincides with the unstable manifold of ${\bm o}$.
\item[(c)]
${\bm f}^{k_0}(U_n)=S_n$ for all $n\ge n_0$ and
\[
{\bm f}^{k_0}(x,y)={\bm R}(x,y)\quad\text{for all }(x,y)\in\left[0,\kappa^{-2n_0}\right]\times I.
\]
\end{enumerate}
In particular, for every $n\ge n_0$,
\[
{\bm f}^{n+k_0}(x,y)=\left(a+b-\kappa ^ny,\kappa ^{-2n}x\right)\in S_{2n}\quad\text{for all }(x,y)\in S_n.
\]
\end{proposition}

In \cite[Theorem 1.3]{KLNS}, all points in boxes $S_n$ for $n\ge n_0$ are shown to be contained in the Lyapunov irregular set although they are \emph{not} contained in the Birkhoff irregular set at all.
By Theorem \ref{main}, random dynamical systems consisting of ${\bm f}$ and physical noises have no Lyapunov irregular set in the sense of Lebesgue measure.
In addition, we will show that the following particular impulsive type perturbation, which is not physical, makes ${\bm f}$ have no Lyapunov irregular set of positive Lebesgue measure  in $\bigcup _{n\geq n_0} S_n$. 
For the setup of our perturbation, 
we let  $\Omega=\{n\in\mathbb{N}\mid n\ge n_0\}  ^{\mathbb{N}_0}$, and
\begin{equation}\label{eq5}
{\bm f}_{\omega}  \coloneqq \bm I_{\omega _0} \circ \bm f \quad \text{for  $\omega=(\omega _0, \omega _1,\ldots )\in\Omega$},
\end{equation}
where ${\bm f}$ is the surface map in Proposition \ref{propGGS} and
\[
\bm I_{ \omega _0}(x,y) =
\begin{cases}
(x,y) &((x,y)\not\in S_n\text{ for any }n\ge n_0)\\
\left(x, \kappa^{n-\omega _0}y\right) &((x,y)\in S_n\text{ for some }n\ge n_0)
\end{cases}
\]
and we consider the  random iteration 
\[
\bm f^n_{\omega}\coloneqq \bm f_{\sigma ^{n-1}\omega} \circ \bm f_{\sigma ^{n-2}\omega} \circ \cdots \bm f_\omega  =(\bm I_{\omega _{n-1}} \circ \bm f) \circ (\bm I_{\omega _{n-2}} \circ \bm f)  \circ \cdots \circ (\bm I_{\omega _0} \circ \bm f),
\]
where $\sigma$ is the full shift on $\Omega$. 
Namely, our random dynamical system is  position-dependent and the jump from $S_n$ into $S_{\omega_0}$ occurs only when an orbital point hits $S_n$ for some $n\ge n_0$.
See Figure \ref{fig2} below.
Such an impulsive noise  reminds us of the  noise considered in \cite[Section 12]{A}, where the perturbation is 
given only around a homoclinic tangency (compare the figure below  with Figure 6 in  \cite{A}).
However, although the  noise in  \cite[Section 12]{A} was physical, our noise \eqref{eq5} is not physical. 
Indeed,
$\left\{ \bm f^n_{\omega} (z) \mid \omega \in \Omega \right\}$ is at most  countable for all $z$ and $n$ by construction, 
so it does not  contain an open ball, that is, (any natural version of) Condition (A) in Definition \ref{d:phy} is violated.

\begin{figure}[ht]
\begin{center}
\includegraphics[width=14cm,bb=0mm 0mm 157mm 75mm,clip]{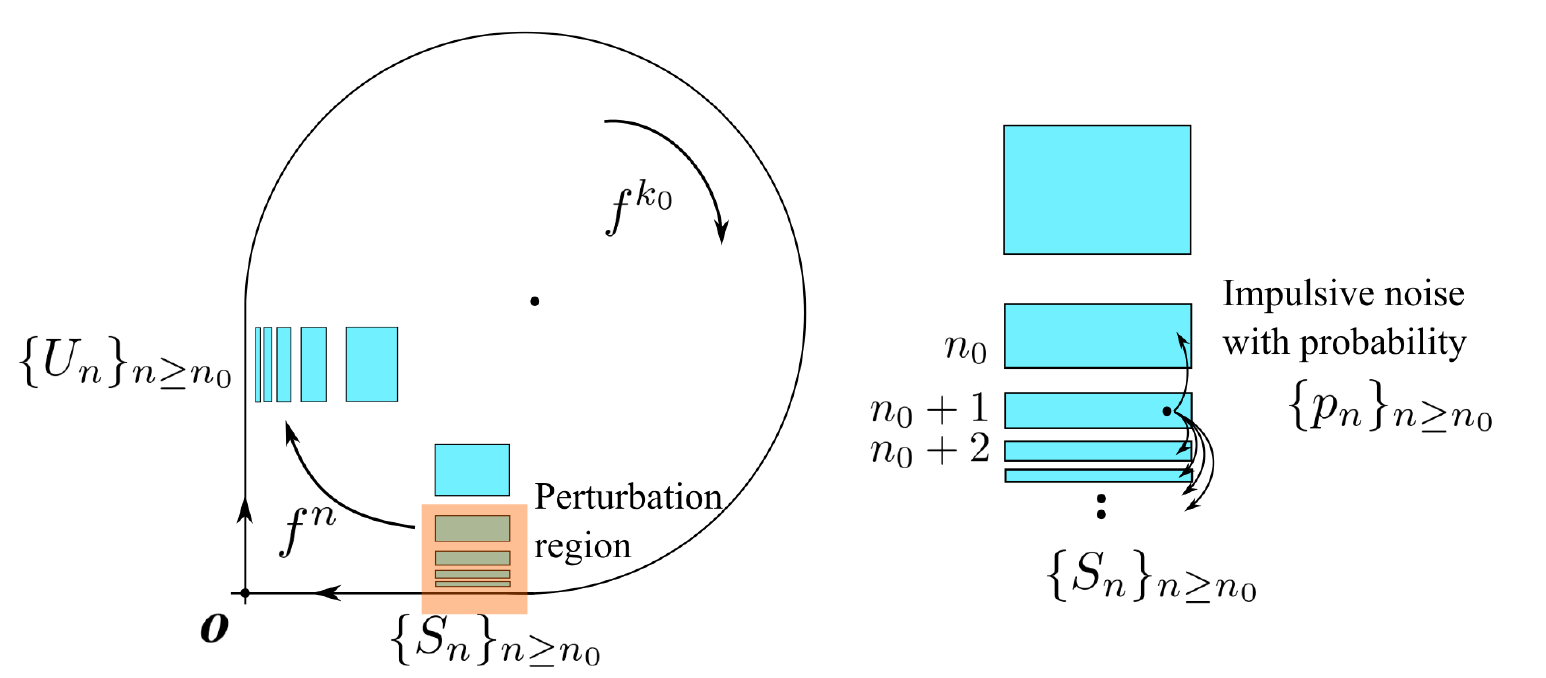}
\end{center}
\caption{The image of the model with perturbation of impulsive type.}
\label{fig2}
\end{figure}

We further suppose that $\Omega$ is equipped with a Bernoulli measure $\mathbb{P}$ such that $\mathbb{P}(\{\omega\in\Omega\mid\omega_0=n\})=p_n$ for $n\ge n_0$ where $p_n$'s are nonnegative and $\sum_{n\ge n_0}p_n=1$ and $\bar{n}\coloneqq \sum_{n\ge n_0}np_n<\infty$.
Then, in contrast to the deterministic case \cite[Theorem 1.3]{KLNS}, we have the unique Lyapunov exponent on $\bigcup_{n\ge n_0}S_n$ for this random dynamical system as follows.

\begin{theorem}\label{propA}
Let ${\bm f}_{\omega}$ be given by (\ref{eq5}).
Then for $\mathbb P$-almost every $\omega\in\Omega$ and every  $n\ge n_0$, $z\in S_n$,
\[
\lim_{n\to\infty}\frac{1}{n}\log\left\lVert D{\bm f}_{\omega}^n(z)v\right\rVert =0
\]
holds for any nonzero vector $v$. 
Consequently, ${\bm f}_{\omega}$ has no Lyapunov irregular set of positive Lebesgue measure in $\bigcup_{n\ge n_0}S_n$ for $\mathbb P$-almost every $\omega\in\Omega$.
\end{theorem}

\begin{proof}
Let $z\in S_{n_1}$ for some $n_1\ge n_0$.
For $\omega\in\Omega$, set $N(1)=n_1+k_0$ and $n_2\equiv n_2(\omega ) =\left(\sigma^{N(1)-1}\omega\right)_0$, so that 
\[
{\bm f}^{N(1)}_{\omega}(z)={\bm f}_{\sigma^{N(1)-1}\omega}\circ \cdots\circ{\bm f}_{\omega}(z)= \bm I_{(\sigma^{N(1)-1}\omega )_0}\circ \bm f^{N(1)} (z)\in S_{n_2}
\]
since the orbit of $z$ by $\bm f$ first returns to $\bigcup _{n\geq n_0} S_n$ at the time $N(1)$ and ${\bm f}^{N(1)}(z)\in S_{2n_1}$ jumps into  ${\bm f}^{N(1)}_{\omega}(z)$ by the impulsive noise $\bm I_{\left(\sigma^{N(1)-1}\omega\right)_0} = \bm I_{n_2}: (x',y')\mapsto (x', \kappa ^{2n_1 - n_2}y')$ on $S_{2n_1}$. 
Then we inductively define $N(k)\equiv N(k;\omega )=\sum_{i=1}^k(n_i+k_0)$ and $n_{k+1}\equiv n_{k+1}(\omega ) =\left(\sigma^{N(k)-1}\omega\right)_0$, so that 
\[
{\bm f}^{N(k)}_{\omega}(z)
= {\bm f}_{\sigma^{N(k)-1}\omega}
\circ\cdots\circ{\bm f}_{\sigma^{N(k-1)}\omega}\circ{\bm f}_{\omega}^{N(k-1)}(z)
= \bm I_{(\sigma^{N(k)-1}\omega )_0}\circ \bm f^{n_k + k_0} \circ{\bm f}_{\omega}^{N(k-1)}(z)
\in S_{n_{k+1}}
\]
since $\bm f^{n_k +k_0} \circ{\bm f}_{\omega}^{N(k-1)}(z) \in S_{2n_k}$ jumps into ${\bm f}^{N(k)}_{\omega}(z)$ by the impulsive noise $\bm I_{(\sigma^{N(k)-1}\omega )_0} = \bm I_{n_{k+1}}: (x',y')\mapsto (x', \kappa ^{2n_k - n_{k+1}}y')$ on $S_{2n_k}$, where  $N(0)\coloneqq  0$.
Hence we have
\[
D\bm f_{\omega}^{N(k)}(z) =\begin{pmatrix}
1 & 0 \\
0 & \kappa ^{2n_k -n_{k+1}}
\end{pmatrix}
\begin{pmatrix}
0 & -\kappa^{n_k}\\
\kappa^{-2n_{k}} & 0
\end{pmatrix}D\bm f_{\omega}^{N(k-1)}(z)
=\begin{pmatrix}
0 & -\kappa^{n_k}\\
\kappa^{-n_{k+1}} & 0
\end{pmatrix}D\bm f_{\omega}^{N(k-1)}(z),
\]
and by induction we get
\begin{eqnarray*}
D{\bm f}_{\omega}^{N(2k-1)+\ell}(z)&=&
\begin{pmatrix}
0 & (-1)^k\kappa^{n_1-2\ell}\\
(-1)^{k-1}\kappa^{-n_{2k}+\ell} & 0
\end{pmatrix}
\quad\qquad \text{for }0\le\ell\le n_{2k},\\
D{\bm f}_{\omega}^{N(2k)+\ell}(z)&=&
\begin{pmatrix}
(-1)^k & 0\\
0 & (-1)^{k}\kappa^{n_1-n_{2k+1}+\ell}
\end{pmatrix}
\quad\quad\quad\quad\text{for } 0\le\ell\le n_{2k+1}.
\end{eqnarray*}
From this, for any nonzero vector $v=\begin{pmatrix} v_1\\v_2\end{pmatrix}$, integers $k\ge 1$ and $0\le\ell< n_{2k}$, we have
\begin{eqnarray*}
\log\left\lVert D{\bm f}_{\omega}^{N(2k-1)+\ell}(z)v\right\rVert
&=&\log\left(\kappa^{2(n_1-2\ell)}v_1^2+\kappa^{2(-n_{2k}+\ell)}v_2^2\right)^{\frac{1}{2}}\\
&&\begin{cases}
\le \log\left(\kappa^{\max\{2n_1,1\}}(v_1^2+v_2^2)\right)^{\frac{1}{2}}=
\log \kappa^{ n_1} + \log \Vert v\Vert 
\\
\ge\log\left(\kappa^{\min\{2(n_1-2n_{2k}),-2n_{2k}\}}(v_1^2+v_2^2)\right)^{\frac{1}{2}}\ge
 \log \kappa^{-2n_{2k}} + \log \Vert v\Vert.
\end{cases}
\end{eqnarray*}
From the construction of ${\bm f}$ (recall the condition (c) in Proposition \ref{propGGS}), there exists a positive number $C$ such that for  $n_{2k}\le \ell\le n_{2k}+k_0$,
\[
\log\left\lVert D{\bm f}_{\omega}^{N(2k-1)}(z)v\right\rVert-\log C\le \log\left\lVert D{\bm f}_{\omega}^{N(2k-1)+\ell}(z)v\right\rVert \le \log\left\lVert D{\bm f}_{\omega}^{N(2k-1)}(z)v\right\rVert+\log C.
\]
Thus, for any nonzero vector $v$, integers $k\geq 1$ and $0\le\ell< n_{2k}+k_0$,  
\[
\log \kappa^{-2n_{2k}} + \log \Vert v\Vert 
-\log C
\le \log\left\lVert D{\bm f}_{\omega}^{N(2k-1)+\ell}(z)v\right\rVert
\le \log \kappa^{ n_1} + \log \Vert v\Vert 
+\log C.
\]
Furthermore,  given $\epsilon >0$ and $\omega\in\Omega$, it holds that
\[
\frac{n_{2k}}{N(2k-1)+\ell } \leq \frac{n_{2k}}{N(2k-1)}
= \frac{\left(\sigma^{N(2k-1)-1}\omega\right)_0}{N(2k-1)}
\leq \limsup_{n\to\infty} \frac{\left(\sigma^{n}\omega\right)_0}{n+1} + \epsilon 
\]
for any   sufficiently large $k$ and any $0\leq \ell \leq n_{2k}+k_0$. Thus,  with similar estimates for $\log\left\lVert D{\bm f}_{\omega}^{N(2k)+\ell}(z)v\right\rVert$, $0\leq \ell \leq n_{2k+1} +k_0$, we have
 \[
 -2\log \kappa \cdot \limsup_{n\to\infty} \frac{\left(\sigma^{n}\omega\right)_0}{n+1}  \leq  \liminf _{n\to \infty} \frac{1}{n} \log\left\lVert D{\bm f}_{\omega}^{n}(z)v\right\rVert  \leq  \limsup _{n\to \infty} \frac{1}{n} \log\left\lVert D{\bm f}_{\omega}^{n}(z)v\right\rVert \leq 0
 \]
for every $\omega\in\Omega$.
On the other hand, since the Koopman operator of $(\sigma,\mathbb{P})$: $\psi\to\psi\circ\sigma$ acts on $L^1(\Omega,\mathbb{P})$ due to the invariance of  $\mathbb P$ for $\sigma$,  by applying the Chacon-Ornstein lemma (Lemma B in \cite[Chapter III]{Fog}) to the Koopman operator of $(\sigma,\mathbb{P})$ and $\psi(\omega)\coloneqq(\omega)_0$ for $\omega\in\Omega$, where $\psi$ is integrable: $\int_{\Omega}\lvert\psi\rvert d\mathbb{P}=\sum_{n\ge n_0}np_i=\bar{n}<\infty$, we get
\[
\lim_{n\to\infty}\frac{(\sigma^{n}\omega)_0}{n+1}=\lim_{n\to\infty}\frac{(\sigma^{n}\omega)_0}{n}=\lim_{n\to\infty}\frac{\psi(\sigma^{n}\omega)}{\sum_{k=0}^{n-1}1_{\Omega}(\sigma^k\omega)}=0
\]
$\mathbb{P}$-almost surely.
In conclusion, we have
\[
\lim_{n\to\infty}\frac{1}{n}\log\left\lVert D{\bm f}_{\omega}^{n}(z)v\right\rVert=0
\]
for $\mathbb P$-almost every $\omega\in\Omega$.
\end{proof}

\begin{remark}
We only use the fact that $(\sigma,\mathbb{P})$ is a probability-preserving system and $\bar{n}$ is finite in the proof of Theorem \ref{propA}.
Hence the result works for any Markov shift over $\Omega$ with these conditions.
Furthermore, the same result (the existence and uniqueness of the Lyapunov exponent under non-physical perturbation)  may  hold   for not only a system with an attracting homoclinic connection but also one with an attracting heteroclinic connection such as the system from \S 4 as long as randomness violates long-time trapped orbits around the connection.
Therefore, Theorem \ref{propA} may offer the key observation to develop Ara\'{u}jo's \cite{A} and our results into non-physical perturbation.
\end{remark}

\section*{Acknowledgments}
This work was partially supported by JSPS KAKENHI Grant Numbers 19K14575, 19K21834 and 21K20330.
 We are sincerely grateful to  Vitor Ara\'ujo and Pablo G.~Barrientos for fruitful discussions and valuable comments.

\end{document}